\documentclass[12pt]{amsart}
\usepackage[dvips]{graphics}
\usepackage[latin1]{inputenc}
\usepackage{amssymb, amsmath, amscd, mathrsfs,marvosym, psfrag}
\input xy
\xyoption{all}
\DeclareMathAlphabet{\mathpzc}{OT1}{pzc}{m}{it}

\newtheorem{theorem}{Theorem}[section]

\newtheorem{corollary}[theorem]{Corollary}
\newtheorem{conjecture}[theorem]{Conjecture}
\newtheorem{lemma}[theorem]{Lemma}

\theoremstyle{definition}
\newtheorem{definition}[theorem]{Definition}

\theoremstyle{remark}

\def\varle{\leqslant}

\newcommand{\CB}{{\mathcal B}}
\newcommand{\CC}{{\mathcal C}}

\newcommand{\CE}{{\mathcal E}}

\newcommand{\CG}{{\mathcal G}}
\newcommand{\CH}{{\mathcal H}}
\newcommand{\CI}{{\mathcal I}}

\newcommand{\CM}{{\mathcal M}}

\newcommand{\CP}{{\mathcal P}}

\newcommand{\CR}{{\mathcal R}}
\newcommand{\CS}{{\mathcal S}}

\newcommand{\CV}{{\mathcal V}}
\newcommand{\CW}{{\mathcal W}}

\newcommand{\Sh}{{\mathcal{S}h}}

\newcommand{\CZ}{{\mathcal Z}}

\newcommand{\LS}{\operatorname{LS}\,}

\newcommand{\SA}{{\mathscr A}}
\newcommand{\SB}{{\mathscr B}}

\newcommand{\SM}{{\mathscr M}}
\newcommand{\SN}{{\mathscr N}}

\newcommand{\fh}{{{\mathfrak h}}}

\newcommand{\fg}{{{\mathfrak g}}} 
\newcommand{\fb}{{{\mathfrak b}}}

\newcommand{\hCG}{{\widehat\CG}}

\newcommand{\hCW}{{\widehat\CW}}

\newcommand{\hbH}{{\widehat\bH}}
\newcommand{\hCP}{{\widehat\CP}}
\newcommand{\hCS}{{\widehat\CS}}

\newcommand{\hS}{{\widehat S}}

\newcommand{\hV}{{\widehat V}}

\newcommand{\hX}{{\widehat X}}

\newcommand{\hw}{\widehat w}

\newcommand{\tS}{\widetilde{S}}

\newcommand{\DC}{{\mathbb C}}

\newcommand{\DR}{{\mathbb R}}
\newcommand{\DZ}{{\mathbb Z}}

\newcommand{\DN}{{\mathbb N}}

\newcommand{\DV}{{\mathbb V}}

\newcommand{\bH}{{\mathbf H}}

\newcommand{\ch}{{\operatorname{char}\, }}

\newcommand{\ind}{{\operatorname{ind}}}

\newcommand{\Spec}{{\operatorname{Spec}}}
\newcommand{\End}{{\operatorname{End}}}

\newcommand{\Hom}{{\operatorname{Hom}}}

\newcommand{\id}{{\operatorname{id}}}

\newcommand{\catmod}{{\operatorname{-mod}}}

\newcommand{\grk}{{{\operatorname{\underline{rk}}}}}

\newcommand{\ol}{\overline}

\newcommand{\linie}{{\,\text{---\!\!\!---}\,}}
\newcommand{\llinie}{{\text{---\!\!\!---\!\!\!---}}}

\newcommand{\ul}{\underline}

\begin{document}

\pagenumbering{arabic}
\title[Lusztig's conjecture as a moment graph problem]{Lusztig's conjecture as a moment graph problem}  
\author[]{Peter Fiebig}
%\thanks{}
\address{Mathematisches Institut, Universit{ä}t Freiburg, 79104 Freiburg, Germany}
\email{peter.fiebig@math.uni-freiburg.de}

\begin{abstract}  We   show that Lusztig's conjecture
  on the irreducible characters of a reductive algebraic group over a field of positive
  characteristic is equivalent to the  generic multiplicity
  conjecture, which gives a formula for the
  Jordan-H\"older multiplicities of baby Verma modules over the
  corresponding Lie algebra. Then we give a
  short overview of a recent proof of  the latter
  conjecture for almost all base fields via the theory of sheaves on moment graphs. 
\end{abstract}

\maketitle

%\tableofcontents

\section{Introduction}
One of the first problems in the representation theory of a reductive
algebraic group is the
determination of its irreducible rational characters. One can
calculate these characters once the irreducible  characters of the
connected, simply connected almost simple algebraic groups are known. In this article we study the
representation theory of such a group.

It is well known that
the irreducible representations are parametrized by their highest
weights, which are the dominant weights  with
respect to the choice of a maximal torus and a Borel subgroup. If the
group is defined over a field $k$ of characteristic zero, the corresponding characters are given by Weyl's
character formula.

Suppose that the characteristic of $k$ is $p>0$. Using the
Steinberg tensor product theorem one can calculate all
irreducible characters from the (finite) set of characters
corresponding to {\em restricted}
highest weights.
If $p$ is at least the Coxeter
number of the group, Lusztig stated in \cite{MR604598} a conjecture about the
irreducible characters for all highest weights that appear  inside the
 Jantzen region. This conjecture was later generalized by Kato in
 \cite{MR772611} to all
restricted weights. 

Lusztig outlined a program for the proof of the conjecture which  was successfully
carried out in a combined effort by Kashiwara \& Tanisaki (\cite{MR1317626}),
Kazhdan \& Lusztig (\cite{KLequiv}) and Andersen, Jantzen \& Soergel
(\cite{MR1272539}). Lusztig's program yielded a 
proof  for {\em almost all} characteristics. More precisely, this
means that for a fixed split almost simple  group scheme $G_\DZ$ over $\DZ$ there exists a number $N$
such that the conjecture holds for the group $G_k=G_\DZ\otimes_\DZ k$ if the characteristic of $k$ is at least $N$. The number $N$, however, is unknown in all but low rank cases.

In fact, the program proved a formula for Jordan--H\"older
multiplicities of the baby Verma modules for the associated Lie
algebra (often referred to as the {\em generic multiplicity conjecture}). In this article we prove the fact that this formula is
equivalent to Lusztig's original conjecture. We mainly use ideas that
are implicit in \cite{MR772611}, see also \cite{MR933346,MR918844,MR933407,Kan}. 

The papers \cite{Fie05a, Fie05b, Fiebig:math0607501,fiebig-2007} give
a very different approach towards the generic multiplicity conjecture  and a new proof for almost
all characteristics (``almost all'' has here the same meaning as before). In addition, the methods developed in the above
series of papers show that the multiplicity one case  holds for
all characteristics that are bigger than the Coxeter number.

The main idea of our approach is to link the generic multiplicity
conjecture to a similar conjecture for intersection
sheaves on an affine moment graph. These objects are of ``linear
algebraic'' nature and admit a rather simple algorithmic
construction.

The multiplicity conjecture on affine moment graphs is known to hold
in two cases. In \cite{MR1871967} it is shown that in the case $k=\DC$
the space of global sections of  the intersection sheaves encode the intersection cohomology (with complex coefficients) of the affine Schubert varieties  of the associated complex 
group. From this we can deduce the conjecture in the case $k=\DC$. From the algebraicity and finiteness of the construction the construction then follows for all fields $k$ with large enough characteristic. 

Secondly, the article \cite{Fiebig:math0607501} contains a
characterization of the {\em $p$-smooth locus} of the moment graph
for all $p$ that are bigger or equal to the Coxeter number. This
result proves the multiplicity one case of the generic multiplicity conjecture.  The arguments used in the above mentioned  paper are rather elementary, and in particular do
not refer to the topology of affine Schubert varieties.

\subsection{Contents}
In  Section \ref{sec:sim}  we recall the main
definitions and results of the representation theory of algebraic
groups and state Lusztig's conjecture. In Section \ref{sec:gmod} we
consider the generic multiplicity conjecture and prove its
equivalence to Lusztig's conjecture.  In Section \ref{sec:mom} we introduce sheaves on moment
graphs and state a multiplicity conjecture for intersection
sheaves. Then we list the proven instances. Finally, we recall the main steps of the proof (which is contained
in \cite{fiebig-2007}) that the generic multiplicity conjecture is
implied by its moment
graph analog.   

\subsection*{Acknowledgements} I wish to thank Geordie Williamson for
various helpful remarks on an earlier version of the paper. 

\section{Simple $G$-modules  and Lusztig's conjecture}\label{sec:sim}
In this section we collect the basic results in the rational
representation theory of an algebraic group and state Lusztig's
conjecture. The most comprehensive reference for the following is
Jantzen's book \cite{MR2015057}. 

Let $k$ be an algebraically closed field of arbitrary characteristic.
It is known how to obtain the irreducible characters of any reductive algebraic
group over $k$ from the irreducible characters of the almost simple,
simply connected algebraic groups over $k$. So let us choose such
a group and let us denote it by $G$.

We choose a maximal torus $T$  in $G$ and let
$X:=\Hom(T,k^\times)$ be its  character lattice. Let $R\subset X$ be the root system of $G$ with
respect to $T$. Denote by $X^\vee:=\Hom(k^\times, T)$ the cocharacter
lattice, and for $\alpha\in R$ let $\alpha^\vee\in X^\vee$ be the
associated coroot.

Choose a Borel subgroup $B\subset G$ that contains $T$
and denote by $R^+\subset R$ its set of roots. Set $R^-=-R^+$ and
denote by $B^-\subset G$ the Borel subgroup that contains $T$ and is opposite
to $B$, so that $R^-$ is its set of roots. Let $\Pi\subset R^+$ be the set of simple roots. The set of dominant weights is
$$
X^+=\{\lambda\in X\mid \langle\lambda,\alpha^\vee\rangle\ge 0\text{
  for all $\alpha\in \Pi$}\},
$$
where $\langle\cdot,\cdot\rangle\colon X\times X^\vee\to \Hom(k^\times,k^\times)=\DZ$ denotes
the natural pairing.

\subsection{Simple $G$-modules}
For $\lambda\in X^+$ we define the {\em Weyl
  module} 
$$
H^0(\lambda):=\ind_{B^-}^G\, k_\lambda=\left\{f\colon G\to k\left| \begin{matrix}
  \text{ $f$ is regular and $f(bg)=\lambda(b^{-1})f(g)$ }\\  \text{ for
    all $b\in B^-$, $g\in G$}\end{matrix}\right\}\right..
$$
Note that $G$ acts on $H^0(\lambda)$ by $(gf)(h)=f(hg)$ for all $f\in
H^0(\lambda)$, $g,h\in G$. We denote by $L(\lambda)$ the socle of $H^0(\lambda)$, i.e.~the maximal semisimple submodule.   Part (1) of the following theorem is due to Chevalley and part (2) is the classical Borel-Weil-Bott theorem.

\begin{theorem} \label{theorem-ChevWeyl}
\begin{enumerate}
\item  Each $L(\lambda)$ for $\lambda\in X^+$ is simple,
and each rational simple $G$-module is isomorphic to $L(\lambda)$
for a unique $\lambda\in X^+$. 
\item If $\ch k=0$, then $L(\lambda)=H^0(\lambda)$.
\end{enumerate}
\end{theorem}
So we achieved for all fields $k$ a classification of the irreducible
rational representations of $G$.

\subsection{Irreducible characters of $G$}
 
Each algebraic torus $T$ over $k$ is a {\em diagonalisable} algebraic
group, which means that for all
rational representations $V$ of $T$ we have a decomposition
$$
V=\bigoplus_{\lambda\in X}V_\lambda,
$$
where $V_\lambda=\{v\in V\mid t.v=\lambda(t)v\quad\forall t\in
T\}$. We denote by 
$$
[V]:=\sum_{\lambda\in X}\dim_k\, V_\lambda\cdot
e^{\lambda}\in\DZ[X]=\bigoplus_{\lambda\in X}\DZ\cdot e^\lambda
$$
the {\em formal character} of $V$. If $V$ is a $G$-module, then we denote by $[V]$ the character of the $T$-module obtained by restriction.

We are interested in the characters of the simple modules
$L(\lambda)$. Recall that the Weyl group $\CW=N_G(T)/T$ of $G$ acts on
$T$ and hence on its character lattice $X$. Denote by $l\colon\CW\to\DN$ the length function
corresponding to our choice of positive roots. Define
$\rho:=1/2\sum_{\alpha\in R^+}\alpha$. We then have $\langle
\rho,\alpha^\vee\rangle = 1$ for all simple roots $\alpha \in \Pi$,
hence $\rho\in X^+$.  

Kempf's vanishing theorem implies that the
character of $H^0(\lambda)$ is given by Weyl's character formula:

\begin{theorem} For each $\lambda\in X^+$ we have
$$
[H^0(\lambda)]=\chi(\lambda):=\frac{\sum_{w\in\CW}(-1)^{l(w)}e^{w(\lambda+\rho)}}{\sum_{w\in\CW}(-1)^{l(w)}e^{w(\rho)}}.
$$
\end{theorem}
Together with Theorem \ref{theorem-ChevWeyl} we obtain a formula for all irreducible
characters in characteristic zero: 
\begin{corollary} If $\ch k=0$, then
$$
[L(\lambda)]=\chi(\lambda)
$$
for $\lambda\in X^+$.
\end{corollary}
The problem of calculating the characters of irreducible
representations in the case $\ch k=p>0$  is much more complicated and it is still not yet completely solved.

\subsection{The affine Weyl group}

From now on suppose that the ground field  $k$ is of characteristic $p>0$.
For $\alpha\in R^+$ and $n\in\DZ$ let us define the affine
transformation $s_{\alpha,n}$ on the lattice $X$ by 
$$
s_{\alpha,n}(\lambda):=\lambda-(\langle\lambda,\alpha^\vee\rangle-pn)\alpha
$$
for all $\lambda\in X$.
Denote by $\hCW$ the {\em affine Weyl group} of our data, i.e.~the group of
affine transformations on $X$ generated by all $s_{\alpha,n}$ with
$\alpha\in R^+$ and $n\in\DZ$. The finite Weyl group $\CW$
appears as the subgroup of $\hCW$ generated by the reflections
$s_{\alpha,0}$ with $\alpha\in R^+$.
The set of {\em simple affine reflections} in  $\hCW$ is $\hCS:=\CS\cup\{s_{\gamma,1}\}$, resp., where $\gamma\in R^+$ is such that $\gamma^\vee$ is the highest 
coroot. Then $(\hCW,\hCS)$ is a  {\em Coxeter system}, hence comes equipped with a length function $l\colon \hCW\to \DN$ and a Bruhat order ``$\varle$'' on $\hCW$.

For $\eta\in X$ we define the translation $t_\eta\colon X\to X$ by
$t_\eta(\lambda)=\lambda+p\eta$.  Note that for $\alpha\in R^+$ the map $s_{\alpha,1}\circ s_{\alpha,0}$
is the translation by $p\alpha$, so $\hCW$ contains $t_\eta$ for all $\eta\in\DZ R$. 

We denote by $\hCW\times X\to X$, $(w,\lambda)\mapsto w\cdot \lambda:=w(\lambda+\rho)-\rho$
the $\rho$-shifted action of $\hCW$ on $X$.
Then the point $0\in X$ is a regular point for the $\rho$-shifted
action of $\hCW$ on $X$ (i.e.~ $w\cdot 0=0$ only if $w=e$), if and only if the prime $p$ is at least the
Coxeter number 
$$
h:=\max_{\alpha\in R^+}\{\langle\rho,\alpha^\vee\rangle+1\}.
$$
This is our general assumption from now on.

\subsection{Lusztig's conjecture}
Let $\hCW^{\text{res,+}}\subset\hCW$ be the set of {\em dominant restricted} elements, i.e.~
$$
\hCW^{\text{res,+}}:=\left\{w\in\hCW\mid 0\le\langle w\cdot 0 , \alpha^\vee\rangle<
p\text{ for all simple roots $\alpha$}\right\}.
$$
%Nach Kato, Seite 127
We will also need the following subsets of $\hCW$. Denote by
$w_0\in\CW$ the longest element and let 
$$
\hCW^{\text{res,$-$}}:=w_0\hCW^{\text{res,$+$}}
$$
be the set of {\em antidominant restricted} elements. Denote by
$\hw_0\in\hCW^{\text{res,$-$}}$ the longest element and  set
$$
\hCW^{\circ}:=\left\{w\in\hCW\mid w\le \hw_0\right\}.
$$
Here ``$\varle$'' denotes the Bruhat order, so $\hCW^{\circ}$ is a finite set.
Finally, denote by
$$
\hCW^{\circ,+}:=\left\{w\in\hCW^{\circ}\mid 0<\langle w\cdot
0+\rho,\alpha^\vee\rangle\text{ for all simple roots $\alpha$}\right\}
$$
the {\em dominant} elements in $\hCW^\circ$.

It is known how to deduce the characters of all simple $G$-modules from
the characters of the modules $L(w\cdot 0)$ with
$w\in\hCW^{\text{res,$+$}}$ using  the  linkage and
translation principles and Steinberg's tensor product theorem. For these
 Lusztig conjectured in \cite{MR604598}
the following formula:
\begin{conjecture}\label{Lconj} For $w\in\hCW^{\text{res,$+$}}$ we have 
$$
[L(w\cdot 0)] 
=\sum_{x\in \hCW^{\circ,+}}(-1)^{l(w)-l(x)}h_{w_0x,w_0w}(1)\chi(x\cdot 0),
$$
where $h_{a,b}(v)\in\DZ[v]$ denotes the Kazhdan-Lusztig polynomial for the
Coxeter system $(\hCW,\hCS)$ at the parameters $a,b\in\hCW$ (cf.~\cite{MR560412}).
\end{conjecture}
(We recall the definition of $h_{a,b}$ in Section \ref{poly}.)
In fact, the above conjecture is slightly more general than the
conjecture of Lusztig and was first stated by Kato in \cite{MR772611}.

\section{From $G$-modules to $\fg$-modules}\label{sec:gmod}
In the following we recall some results of the theory of restricted
representations of the Lie algebra of $G$. These can be found, stated in the 
equivalent language of $G_1T$-modules, in Jantzen's book \cite{MR2015057}.

Let $\fh\subset \fb\subset \fg$ be the Lie algebras of $T\subset
B\subset G$.  Since $\ch k=p>0$, these Lie algebras 
are $p$-Lie algebras, i.e.~they come equipped with a ``$p$-th power'' map
$X\mapsto X^{[p]}$. (By definition, $\fg$ is the Lie algebra of derivations in $\End_k(k[G])$ that commute with the left translation action of $G$, and in characteristic $p$, the
$p$-th power of such a derivation is again in $\fg$.)

Let $U=U(\fg)$ be the universal enveloping algebra and let $
U^{\text{res}}:=U/I$, 
where $I\subset U$ is the twosided ideal generated by the elements $X^p-X^{[p]}$ with $X\in\fg$,
be the {\em restricted enveloping algebra} of
$\fg$. The adjoint action of $T$ on $\fg$ yields an $X$-grading on the
spaces $\fg$, $U$ and $U^{\text{res}}(\fg)$. Denote by $U^{\text{res}}\catmod_X$ the category of $X$-graded
$U^{\text{res}}$-modules $M=\bigoplus_{\nu\in X}M_\nu$ and define the full
subcategory 
$$
\CC:=\left\{M\in U^{\text{res}}\catmod_X\left|\, \begin{matrix} \text{$M$ is
        finitely generated, } \\
Hm=\ol\nu(H) m \\
\text{ for all $H\in\fh$, $\nu\in X$, $m\in M_\nu$}
\end{matrix}
\right.\right\},
$$
where for $\nu\in X=\Hom(T,k^\times)$ we denote by $\ol \nu\in\Hom(\fh,k)$ its
differential. Note that differentiation of rational $G$-modules naturally yields objects in
$\CC$ if we remember the $X$-grading given by the action of $T$.

We define the {\em  formal character} of an object $M$ of $\CC$ by
$$
[M]:=\sum_{\nu\in X}\dim_k M_\nu \cdot e^\nu\in\DZ[X].
$$

\subsection{Standard objects and simple objects in $\CC$}
For $\lambda\in X$ define the {\em standard} or {\em baby
  Verma} module  by
$$
Z^\prime(\lambda):=U^{\text{res}}(\fg)\otimes_{U(\fb)}\, k_\lambda.
$$ 
Here $k_\lambda$ is the one-dimensional $X$-graded $\fb$-module sitting in degree
$\lambda$ on which $\fb$ acts via the character $\ol\lambda$ (here we
use the inverse of the isomorphism $\fh\stackrel{\sim}\to\fb/[\fb,\fb]$).
Note that $Z^\prime(\lambda)$ is an object in $\CC$, where the $X$-grading
is induced by the $X$-gradings on $U^{\text{res}}$ and on $k_\lambda$. Let  us denote by $L^\prime(\lambda)$ the maximal semisimple quotient of $Z^\prime(\lambda)$.
The induced $X$-grading gives $L^\prime(\lambda)$ the structure of an
object in $\CC$. We then have:

\begin{theorem} For each $\lambda\in X$ the module $L^\prime(\lambda)$ is
  simple and each simple object in $\CC$ is isomorphic to
  $L^\prime(\lambda)$ for a unique $\lambda\in X$. 
\end{theorem}

Now define the set of {\em restricted dominant weights} by
$$
X^+_p:=\{\lambda\in X\mid 0\le \langle\lambda,\alpha^\vee\rangle<
p\text{ for all simple roots $\alpha$}\}.
$$
%Nach Andersen, ARCATA

\begin{theorem}\label{Lstayssimple} If $\lambda\in X^+_p$, then 
$$
L^\prime(\lambda)\cong \ol{L(\lambda)}
$$
(here we write $\ol{L(\lambda)}$ for the object in $\CC$ obtained from the
corresponding rational $G$-module by differentiation). In particular, we have
$$
[L^\prime(\lambda)]=[\ol{L(\lambda)}].
$$
for all $\lambda\in X^+_p$.
\end{theorem}

Note that for $w\in\hCW^{\text{res,$+$}}$ we have $w\cdot 0\in X_p^+$. Hence,  in order to prove Lusztig's conjecture it is enough to calculate the  characters of the simple objects in $\CC$. 

The characters of the standard modules $Z^\prime(\mu)$ are easy to
obtain from their explicit construction: we have 
\begin{eqnarray*}
[Z^\prime(\mu)] & = & e^{\mu}\prod_{\alpha\in
  R^+}(1+e^{-\alpha}+\dots+e^{-(p-1)\alpha}) \\
& = & e^{\mu}\prod_{\alpha\in R^+}
\frac{1-e^{-p\alpha}}{1-e^{-\alpha}}.
\end{eqnarray*}
In order to calculate the formal characters of $L^\prime(\lambda)$ for
$\lambda\in X$ it is sufficient to know all Jordan-Hölder multiplicities
$[Z^\prime(\mu):L^\prime(\lambda)]$. In Section \ref{sec:genconj} we
state a conjecture for this multiplicity (known as the {\em generic
  multiplicity conjecture}, or the {\em $G_1T$-conjecture}) and
prove its equivalence to Lusztig's conjecture.

\subsection{Various polynomials}\label{poly}

We denote by $\hbH$ the {\em affine Hecke algebra}, i.e. the Hecke
algebra associated to the Coxeter system $(\hCW,\hCS)$. Recall that
$\hbH$ is the free $\DZ[v,v^{-1}]$-module with basis $\{T_x\}_{w\in
  \hCW}$ endowed with the unique $\DZ[v,v^{-1}]$-bilinear multiplication such that
\begin{eqnarray*}
{ T}_x\cdot { T}_{y} & = & { T}_{xy}\quad\text{if $l(xy)=l(x)+l(y)$}, \\
{ T}_s^2 & = & v^{-2}T_e+(v^{-2}-1)T_s \quad\text{for $s\in\hCS$}.
\end{eqnarray*}
Then ${ T}_e$ is a unit in $\hbH$ and for any $T_x$ with $x\in\hCW$
there exists a multiplicative inverse ${ T}^{-1}_x$ in $\hbH$. For
example, for $s\in\hCS$ we have ${ T}_s^{-1}=v^2T_s+(v^2-1)$. There is a duality (i.e.~ a $\DZ$-linear anti-involution) $d\colon\hbH\to\hbH$, given by $d(v)=v^{-1}$ and $d(T_x)  =  T_{x^{-1}}^{-1}$ for $x\in\hCW$. 

Now set $ H_x:=v^{l(x)}T_x$. We quote the following result of \cite{MR560412} in the language and normalization of \cite{MR1445511}.
\begin{theorem}\label{self-dual elts} 
For any $w\in\hCW$ there exists a unique element $\ul{H}_w$ with $d(\ul{H}_w)=\ul{H}_w$ and $\ul H_w\in H_w+\sum_{x\in\hCW} v\DZ[v] H_x$. 
\end{theorem} 
For example, we have $\ul{H}_e=H_e$ and $\ul{H}_s=H_s+v H_e$ for each $s\in\hCS$.  We define the polynomials $h_{x,w}\in \DZ[v]$ for $w,x\in \hCW$ by the equations
$$
\ul H_w=\sum_{x\in \hCW} h_{x,w} H_x.
$$
It turns out that $h_{w,w}=1$ and that $h_{x,w}\ne 0$ implies $x\le w$. 
The usual Kazhdan--Lusztig polynomials $P_{x,w}$ (cf.~ \cite{MR560412}) then satisfy
$
h_{x,w}=v^{l(w)-l(x)}P_{x,w}$.
In particular, we have $h_{x,w}(1)=P_{x,w}(1)$.

We need two more sets of polynomials, the {\em generic} polynomials
$q_{A,B}$ and the {\em periodic} polynomials $p_{A,B}$. Both  are
indexed by pairs of {\em alcoves}. So let us introduce alcove geometry. 

Consider the real vector space $V=X\otimes_{\DZ}\DR$. It is
acted upon by $\hCW$ and for each
$\alpha\in R^+$, $n\in\DZ$ the affine hyperplane  
$$
H_{\alpha,n}=\{v\in V\mid \langle v,\alpha^\vee\rangle=pn\}
$$
is the set of fixed points for $s_{\alpha,n}$. Denote by $\SA$ the set of connected components of the complement of the union of all
reflection hyperplanes $H_{\alpha,n}$ in $V$.  An element $A\in \SA$ is called an {\em alcove}. The affine Weyl group $\hCW$ naturally acts on $\SA$, but  note that the translations $t_\alpha$ with $\alpha\in X$ act on $\SA$ as
well. Let 
$$
C:=\{v\in V\mid \langle v,\alpha^\vee\rangle >0 \text{ for all simple
  roots $\alpha$}\}
$$
be the {\em  dominant cone}, and denote by $A^+\subset C$ the unique alcove
that contains $0\in V$ in its closure. Then the map
\begin{eqnarray*}
\hCW & \to & \SA, \\
w& \mapsto & A_w:=w(A^+)
\end{eqnarray*}
is a bijection that we fix from now on. We denote the inverse map by
$A\mapsto w_A$. 

For convenience we now reparametrize the polynomials $h_{x,y}$: set
$$
h_{A,B}:=h_{w_0w_A,w_0w_B}
$$
for all $A,B\in\SA$. (Recall that $w_0\in \CW$ is the longest element.) Now let us fix two alcoves  $A$ and $B$  and consider for each $\lambda\in X_+$ the polynomial
$h_{t_{\lambda}(A), t_{\lambda}(B)}$. It turns out that this polynomial does not depend on $\lambda$ as long as $\lambda$ lies ``deep inside'' $X^+$ (cf. \cite[Theorem 6.1 \& Proposition 3.4]{MR1445511}), so there is a well-defined polynomial
$$
q_{A,B}:=\lim_{\lambda\in X^+} h_{t_{\lambda}(A), t_{\lambda}(B)}.
$$
It is called the {\em generic} polynomial associated to the pair
$(A,B)$.

Now $q_{A,B}\ne 0$ implies that $ h_{t_{\lambda}(A), t_\lambda(B)}\ne 0$
for $n\gg 0$, hence that $t_\lambda w_{A}\le t_{\lambda}w_B$ for $\lambda$ deep inside $X^+$. This
gives rise to a partial order on the set $\SA$: we define
$A\preceq B$ if $t_{\lambda}w_A\le t_{\lambda}w_B$ for $\lambda$ big enough. Hence
$q_{A,B}\ne 0$ implies $A\preceq B$.

There
is a third set of polynomials $p_{A,B}$, the {\em periodic}
polynomials (the original definition is due to Lusztig). They are defined as the coefficients of certain selfdual elements in the {\em periodic Hecke module} of the affine Hecke algebra $\hbH$ (cf. \cite[Theorem 4.3]{MR1445511}).  We could also define them by the following inversion property with respect to the generic polynomials. We have for all pairs $(A,C)$ of alcoves the following identity (cf. \cite[Theorem 6.1]{MR1445511}):
$$
\sum_{B\in\SA} (-1)^{d(A,B)} q_{w_0(B), w_0(A)}\cdot p_{B,C}=\delta_{A,C}.
$$
Here  $(-1)^{d(A,B)}$ denotes the parity of the number of reflection hyperplanes separating
the alcoves $A$ and $B$. We have
$(-1)^{d(A,B)}=(-1)^{l(w_A)-l(w_B)}$. Moreover, for fixed $C$ we have $p_{B,C}\ne
0$ only for finitely many $B$, so the sum above is finite.

\subsection{The generic multiplicity conjecture}\label{sec:genconj}

Now we can state the multiplicity conjecture for the standard modules
$Z^\prime(x\cdot 0)$.

\begin{conjecture}\label{genconj} For each $x,w\in\hCW$ we have
$$
[Z^\prime(x\cdot 0):L^\prime(w\cdot 0)]=p_{A_{w_0x},A_{w_0w}}(1).
$$
\end{conjecture}
Due to an inherent $pX$-symmetry in the category $\CC$ it is enough
to prove this conjecture for elements $w\in\hCW^{\text{res,$-$}}$ or $w\in\hCW^{\text{res,$+$}}$.

The following statement is well-known, but I could not find a proof
in the literature. The main ideas of the proof  provided below
are due to Kato (cf.~ \cite[Theorem 3.5]{MR772611}).

\begin{theorem} Conjecture  $\ref{genconj}$ is equivalent to Lusztig's
  conjecture  $\ref{Lconj}$. 
\end{theorem}
\begin{proof}  We define a partial order on $X$ by setting $\lambda\le
  \mu$ if $\mu-\lambda$ is a sum of positive roots. Using this order
  we get a completion  of the group ring
  $\DZ[X]$: define
  $\widehat{\DZ[X]}\subset\prod_{\lambda\in X} \DZ\cdot e^{\lambda}$
  as the subset of elements $\sum_{\lambda\in X}a_\lambda\cdot
  e^{\lambda}$ for which there exists some $\mu\in X$ such that
  $a_\lambda\ne 0$ implies $\lambda\le \mu$. Note that the convolution
  product is still well defined on $\widehat{\DZ[X]}$ and makes it
  into a ring. 

 Since the periodic polynomials are inverse to the generic polynomials,
Conjecture \ref{genconj} is equivalent to the set of identities 
$$
[L^\prime(w\cdot 0)] = \sum_{z\in\hCW}
(-1)^{l(w)-l(z)}q_{A_z,A_w}(1)[Z^\prime(z\cdot 0)]
$$
for all $w\in\hCW^{\text{res,$+$}}$. (Note that the right hand side is a priori an
element in $\widehat{\DZ[X]}$.)  Using  Theorem
\ref{Lstayssimple} we see that  we have to
show, for all  $w\in\hCW^{\text{res,$+$}}$, that 
$$
\sum_{x\in\hCW^{\circ,+}} (-1)^{l(x)} h_{w_0x,w_0w}(1)\chi(x\cdot 0) =
\sum_{z\in\hCW} (-1)^{l(z)} q_{A_z,A_w}(1)[Z^\prime(z\cdot 0)]\eqno{(\ast)}
$$
in order to prove our claim (we omitted the factor $(-1)^{l(w)}$ on
both sides).

By \cite[Corollary 5.3]{MR591724} we have
$$
h_{w_0x,w_0w}(1)=p_{A_{x},A_{w}}(1),
$$
so the left hand side of equation $(\ast)$ equals
$$
\LS=\sum_{x\in\hCW^{\circ,+}} (-1)^{l(x)}p_{A_{x},A_{w}}(1)\chi(x\cdot 0).
$$
This formula is stated as the {\em $G_1T$-Lusztig conjecture} in \cite[3.3.4]{Kan}.

Now we use Kostant's formula for the Weyl character (note that the
following is an identity in $\widehat{\DZ[X]}$):
$$
\chi( x\cdot 0 )  =   \sum_{y\in\CW}
(-1)^{l(y)}e^{yx\cdot 0}\prod_{\alpha\in R^+}(1-e^{-\alpha})^{-1}.
$$
If we substitute this for  $\chi(x\cdot 0)$ in the
earlier  expression we get
$$
\LS  =
 \sum_{x\in\hCW^{\circ,+}} (-1)^{l(x)} p_{A_{x},A_{w}}(1)\sum_{y\in\CW}
(-1)^{l(y)}e^{yx\cdot 0}\prod_{\alpha\in R^+}(1-e^{-\alpha})^{-1}.
$$
Using the fact that the multiplication
$\CW\times\hCW^{\circ,+}\to\hCW^{\circ}$, $(y,x)\mapsto z=yx$, is a
bijection, that $l(z)=l(y)+l(x)$ in this situation and that
$p_{A_{yx},A_w}(1)=p_{A_x,A_w}(1)$ for all $y\in\CW$ we arrive at
$$ 
\LS =  
\sum_{z\in\hCW^{\circ}} (-1)^{l(z)} p_{A_{z},A_{w}}(1)e^{z\cdot
  0}\prod_{\alpha\in R^+}(1-e^{-\alpha})^{-1}.\eqno{(\ast\ast)}
$$

Now consider the free $\DZ[v,v^{-1}]$-module
$\CP:=\bigoplus_{A\in\SA}\DZ[v,v^{-1}]\cdot A$ with basis $\SA$ and its partial
completion
$\hCP\subset\prod_{A\in\SA}\DZ[v,v^{-1}]\cdot A$, that consists of all
elements  with support bounded from above, i.e.~ it is the set of elements $\sum
b_A\cdot A$ for which  there exists an
alcove $B$ such that $b_A\ne 0$ implies $A\preceq B$. We define the
free $\DZ$-modules $\DZ[\SA]$ with basis $\SA$ and its partial
completion $\widehat{\DZ[\SA]}$ likewise.   Evaluation at
$v=1$ gives maps $\CP\to\DZ[\SA]$ and
$\hCP\to\widehat{\DZ[\SA]}$. 

Set
$P_B:=\sum_A p_{A,B}\cdot A\in\CP\subset\hCP$ and $Q_B:=\sum_A
q_{A,B}\cdot A\in\hCP$, and define a $\DZ[v,v^{-1}]$-linear operator 
\begin{eqnarray*}
\eta\colon\hCP & \to & \hCP \\
A & \mapsto & \left(\prod_{\alpha\in R^+}
(\id+v^2 t_{-p\alpha}+v^4 t_{-2p\alpha}  +\dots)\right)(A).
\end{eqnarray*}

The following theorem relates the periodic polynomials to the generic
polynomials. It is due to Kato (cf.~\cite{MR772611}). We state it in the language of \cite{MR1445511}, cf. Theorem 6.3.
\begin{theorem} We have $\eta(P_B)=Q_B$ for all $B\in\SA$.
\end{theorem}

 If we evaluate Kato's formula at $v=1$ we get the following equation in $\widehat{\DZ[\SA]}$:
\begin{eqnarray*}
\eta(P_B)(1) & = & \sum_{A\in\SA} p_{A,B}(1)\cdot
\left(\prod_{\alpha\in R^+}( \id+ t_{-p\alpha} + t_{-2p\alpha} +\dots)\right)(A) \\
& = & \sum_{A\in\SA} q_{A,B}(1)\cdot A.
\end{eqnarray*}

 Consider the injections  $\DZ[\SA]\to\DZ[X]$ and
 $\widehat{\DZ[\SA]}\to\widehat{\DZ[X]}$ of abelian groups that
map an alcove $A$ to $(-1)^{l(w_A)}e^{w_A\cdot 0}$. We deduce from the above
the following identity in $\widehat{\DZ[X]}$:
$$
\sum_{z\in\hCW}(-1)^{l(z)} p_{A_{z},A_{w}}(1)
\left(\prod_{\alpha\in R^+}(1-e^{-p\alpha})^{-1}\right)e^{z\cdot 0}
= \sum_{z\in\hCW} (-1)^{l(z)} q_{A_z,A_w}(1) e^{z\cdot 0}.
$$

This we use on the identity $(\ast\ast)$ as follows:
\begin{eqnarray*}
\LS & = & \sum_{z\in\hCW^{\circ}} (-1)^{l(z)} p_{A_{z},A_{w}}(1)e^{z\cdot
  0}\prod_{\alpha\in R^+}(1-e^{-\alpha})^{-1} \\
& = & \left(\sum_{z\in\hCW^{\circ}} (-1)^{l(z)} p_{A_{z},A_{w}}(1)e^{z\cdot
  0}\prod_{\alpha\in
  R^+}(1-e^{-p\alpha})^{-1}\right)\prod_{\alpha\in R^+}\frac{1-e^{-p\alpha}}{1-e^{-\alpha}} \\
& = & \sum_{z\in\hCW} (-1)^{l(z)} q_{A_z,A_w}(1)
e^{z\cdot 0}\prod_{\alpha\in R^+}\frac{1-e^{-p\alpha}}{1-e^{-\alpha}} \\
& = & \sum_{z\in\hCW}(-1)^{l(z)} q_{A_z,A_w}(1) [Z^\prime(z\cdot 0)],
\end{eqnarray*}
which is what we claimed. (In step $2$ we used that $p_{A_z,A_w}\ne
0$ implies that $z\in\hCW^{\circ}$ for $w\in\hCW^{\text{res,$+$}}$.)
\end{proof}

We have now translated Lusztig's original conjecture on the characters of simple rational representations of a semisimple algebraic group into a multiplicity conjecture for the Jordan--H\"older series of standard modules over its Lie algebra. We are at the starting point of the work of Andersen, Jantzen and Soergel.
\subsection{The Brauer-Humphreys reciprocity and equivariant projective covers}\label{equivproj}
For each $\lambda\in X$ there exists a projective cover $Q^\prime(\lambda)$
of $L^\prime(\lambda)$ in $\CC$, and each
$Q^\prime(\lambda)$ admits a filtration whose successive subquotients are
isomorphic to various $Z^\prime(\mu)$'s. We denote the corresponding
multiplicity by $(Q^\prime(\lambda):Z^\prime(\mu))$. It is independent of the
particular filtration, and the following reciprocity principle was proven by Humphreys (cf.~\cite{MR0283038}):
$$
[Z^\prime(\mu):L^\prime(\lambda)]=(Q^\prime(\lambda):Z^\prime(\mu)).
$$
In order to prove Conjecture \ref{genconj} it is hence enough to calculate the numbers $(Q^\prime(y\cdot 0):Z^\prime(x\cdot
0))$ for all $x,y\in\hCW$. 

Now we need a certain deformation of the standard and the projective objects. 
For this  we identify the $k$-vector space $X^\vee\otimes_\DZ k$ with
the Lie algebra $\fh$ of the torus $T$.  Let $S=S_k(\fh)$  be the
symmetric algebra of $\fh$ and denote by $\tS$ the completion of $S$ at the maximal ideal generated by $\fh$.  The category $\CC$ then admits an {\em
  equivariant} (or {\em deformed}) version 
$\CC_{\tS}$ of certain $\fg$-$\tilde S$-bimodules (cf.~\cite{MR1272539,MR1357204}).

For each $\lambda\in X$ there is an {\em equivariant} standard module
$Z^\prime_{\tS}(\lambda)\in\CC_{\tS}$ and an indecomposable projective object
$Q^\prime_{\tS}(\lambda)$ in $\CC_{\tS}$ that maps
surjectively onto $Z^\prime_{\tS}(\lambda)$. Moreover,
$$
Z^\prime_{\tS}(\lambda)\otimes_{\tS} k\cong Z^\prime(\lambda),\quad Q^\prime_{\tS}(\lambda)\otimes_{\tS}
k=Q^\prime(\lambda).
$$
Here we used  the ring homomorphism $\tilde S\to k$ corresponding to
the unique closed point of $\Spec\, \tilde S$.

As in the non-equivariant situation, each $Q^\prime_{\tS}(\lambda)$ admits a finite filtration with subquotients
isomorphic to equivariant standard objects, and for the multiplicities
we have
$$
(Q^\prime_{\tS}(\lambda):Z^\prime_{\tS}(\mu))=(Q^\prime(\lambda):Z^\prime(\mu)).
$$

The advantage of the equivariant objects is that they can be thought
of as families  (parametrized by $\Spec\, {\tS}$) of the corresponding
non-equivariant objects. The structure of $Q^\prime_{\tS}(\lambda)$ at
generic points and on codimension one hyperplanes of $\Spec\, {\tS}$ is well
understood, which allows us to gain information on $Q^\prime_{\tS}(\lambda)$ at the
closed point corresponding to the maximal ideal $\fh\cdot {\tS}$. 

In \cite{MR1272539} this is used to construct  a ``combinatorial
category'' that is equivalent to the full subcategory of projective
objects in $\CC_{\tS}$. It is via this result that we are able to link
the category of intersection sheaves on an affine moment graph to
Lusztig's conjecture.  

\section{From intersection sheaves on affine moment graphs to $\fg$-modules}\label{sec:mom}

The most important tool for our approach towards Lusztig's conjecture
is the theory of {\em intersection sheaves on moment graphs}. Let us
recall the main definitions and results following \cite{MR1871967} and \cite{Fie05a}.

\subsection{Moment graphs}
Let $k$ be a field and $W$ a $k$-vector space.
A {\em moment graph} $\CG=(\CV,\CE,\varle,\alpha)$ over $W$ is given by the
following data:
\begin{itemize}
\item $(\CV,\CE)$ is a finite graph with set of vertices $\CV$ and set
  of edges $\CE$ (we assume that $(\CV,\CE)$ has no
  double edges and no loops),
\item ``$\varle$'' is a partial order on $\CV$ such that two elements
  $x,y\in\CV$ are comparable if they lie on a common edge and
\item $\alpha$ is a map from $\CE$ to $W\setminus\{0\}$ called the labelling.
\end{itemize}

For an edge $E=(x,y)$ we usually write $E\colon x\linie y$ or $E\colon
x\to y$ if $x< y$. If
$\alpha=\alpha(E)$ is the label of $E$ we write $E\colon
x\stackrel{\alpha}\llinie y$ or $E\colon x\xrightarrow{\alpha} y$. 

A {\em sub-moment graph}  of $\CG$ will for us always be a full
subgraph endowed with the induced partial order and the induced
labelling. So such a subgraph is determined by  a
subset of $\CV$. In particular, for each vertex $w\in\CV$ we define
the sub-moment graph $\CG_{\le w}$ of $\CG$ that is given by the set $\{x\in\CV\mid
x\le w\}$. 

\subsection{Sheaves on $\CG$}

Denote by $S=S_k(W)$ the symmetric algebra over the $k$-vector space
$W$, and endow it with the algebra-$\DZ$-grading that is given by
setting $W\subset S$ in degree $2$. In the following we will only
consider graded $S$-modules and graded homomorphisms of degree $0$.

A {\em sheaf}
$\SM=\left(\{\SM^x\},\{\SM^E\},\{\rho_{x,E}\}\right)$ on the moment graph $\CG$
consists of the following data:
\begin{itemize}
\item an $S$-module $\SM^x$ for each vertex $x\in\CV$,
\item an $S/\alpha(E)S$-module $\SM^E$ for each edge $E\in\CE$,
\item a homomorphism $\rho_{x,E}\colon\SM^x\to\SM^E$ of $S$-modules for each edge $E$
  that contains $x$ as a vertex.
\end{itemize}
Note that the sheaves on $\CG$ form an $S$-linear category $\Sh(\CG)$ in the
obvious way: a morphism $f\colon\SM\to\SN$ is given by $S$-linear maps
$f^x\colon\SM^x\to\SN^x$ and $f^E\colon \SM^E\to\SN^E$ for all
vertices $x$ and edges $E$ that are compatible with the maps $\rho^\SM$ and $\rho^\SN$.

The space of {\em global sections} of $\SM$ is defined as
$$
\Gamma(\SM):=\left\{(m_x)\in\bigoplus_{x\in\CV}\SM^x\left|\,
\begin{matrix}\rho_{x,E}(m_x)=\rho_{y,E}(m_y) \\
\text{ for all edges $E\colon x\linie
  y$}
\end{matrix}\right\}\right..
$$
For a subgraph $\CI$ of $\CG$ we have the obvious restriction
$\SM|_\CI$ of a sheaf $\SM$ to $\CI$ and we denote by $\Gamma(\CI,\SM):=\Gamma(\SM|_\CI)$ its space of sections on $\CI$.

The {\em structure algebra} of $\CG$ is 
$$
\CZ=\CZ(\CG):=\left\{(z_x)\in\bigoplus_{x\in\CV} S \left|\,
\begin{matrix} z_x\equiv z_y\mod \alpha \\
\text{ for all edges $E\colon x\stackrel{\alpha}\llinie 
  y$}
\end{matrix}\right\}\right..
$$
It is a $\DZ$-graded algebra by coordinatewise
multiplication. Similarly, 
for each sheaf $\SM$ the space $\Gamma(\SM)$ is naturally a
graded $\CZ$-module by coordinatewise multiplication, so $\Gamma$ is a functor from $\Sh(\CG)$ to
$\CZ\catmod$. 

The following property of moment graphs is crucial for our approach
and yields the only restriction on the characteristic of the field $k$ in the proof of the
multiplicity one result in \cite{Fiebig:math0607501}.

\begin{definition} We say that $\CG$ is a {\em
    Goresky-Kottwitz-MacPherson graph} (or a {\em GKM-graph}), if for
  any vertex $x$ of $\CG$ the
  labels of the edges connected to $x$ are pairwise linearly
  independent in $W$.
\end{definition}

\subsection{Intersection sheaves on a moment graph}

Let $x\in\CV$ and define 
\begin{eqnarray*}
\CV_{\delta x} & := & \{y\in\CV\mid x<y\text{ and $y$ is connected to $x$ by an edge}\},\\
\CE_{\delta x} & := & \{E\colon x\to y\in \CE\mid y\in\CV_{\delta x}\}.
\end{eqnarray*}
For a sheaf $\SM$ on $\CG$ and a vertex $x$ define $\SM^{\delta
  x}\subset\bigoplus_{E\in\CE_{\delta x}} \SM^E$ as
the image of the composition
$$
\Gamma(\{y\mid x<y\},\SM)\subset
\bigoplus_{x<y}\SM^y\stackrel{p}\to\bigoplus_{y\in\CV_{\delta x}}\SM^y\stackrel{\bigoplus
  \rho_{y,E}}\to\bigoplus_{E\in\CE_{\delta x}}\SM^E,
$$
where $p\colon
\bigoplus_{x<y}\SM^y\stackrel{p}\to\bigoplus_{y\in\CV_{\delta
    x}}\SM^y$ is the projection along the decomposition (note that
$\CV_{\delta x}\subset \{y\mid x<y\}$). Define the map
$\rho_{x,\delta x}:=(\rho_{x,E})^T_{E\in\CE_{\delta
    x}}\colon\SM^x\to\bigoplus_{E\in\CE_{\delta x}}\SM^E$.
The following theorem characterizes the {\em canonical sheaf} considered by
Braden and MacPherson in \cite{MR1871967}. We call it the {\em
  intersection sheaf}.

\begin{theorem} For each $w\in\CV$ there exists an up to isomorphism
  unique sheaf $\SB(w)=\left(\{\SB(w)^x\},\{\SB(w)^E\},\{\rho_{x,E}\}\right)$ on $\CG$ with the following properties:
\begin{enumerate}
\item $\SB(w)^x=0$ unless $x\le w$, and $\SB(w)^w\cong S$.
\item For each edge $E\colon x\to y$ the map
  $\rho_{y,E}\colon\SB(w)^y\to\SB(w)^E$ is surjective with kernel
  $\alpha(E)\cdot \SB(w)^y$, i.e.~it induces an isomorphism
  $\SB(w)^y/\alpha(E)\cdot \SB(w)^y\cong \SB(w)^E$.
\item For each $x<w$, the image of $\SB(w)^x$ under $\rho_{x,\delta x}$ is contained in
  $\SB(w)^{\delta x}$ and  the map $\rho_{x,\delta x}\colon
  \SB(w)^x\to\SB(w)^{\delta x}$ is a projective cover  in the category of
  graded $S$-modules.
\end{enumerate}
\end{theorem}
By a {\em projective cover} of a graded $S$-module $M$ we mean a
graded free $S$-module $P$ together with a surjective homomorphism
$f\colon P\to M$ of graded $S$-modules such that for any homomorphism
$g\colon Q\to P$ of graded $S$-modules the following holds: If $f\circ
g$ is surjective, then $g$ is surjective.

The properties stated in the theorem yield in fact an
algorithm for the  construction of  the sheaf $\SB(w)$.
In \cite{Fie05a} we give another characterization of $\SB(w)$ as the
universal {\em flabby} sheaf on $\CG_{\le w}$.

\subsection{The affine moment graph associated to the root system}

The moment graph that we connect to Lusztig's conjecture is a certain
finite moment graph associated to the {\em affine} Weyl group. Its labels are  affine coroots. In order to introduce these, we first have to add a copy of $\DZ$ to the coweight lattice, so we set $\hX^\vee:=X^\vee\oplus \DZ$. Consider $X^\vee$ as a subset in $\hX^\vee$ via the
canonical injection, and set $\delta^\vee:=(0,1)\in\hX^\vee$. To the reflection $s_{\alpha,n}\in\hCW$ we associate the {\em affine coroot} 
$$
\alpha^\vee_n:=(\alpha^\vee,-n)=\alpha^\vee-n\delta^\vee \in \hX^\vee.
$$

Now let $k$  be an arbitrary field and set $\hV_k:=\hX^\vee\otimes_\DZ k$.
Define the moment graph
$\hCG_k=(\CV,\CE,\le,\alpha)$ over the vector space $\hV_k$
as follows. Set $\CV=\hCW$ and let ``$\varle$'' be the Bruhat order on
$\hCW$. Connect $x,y\in\hCW$ by an edge $E$ if
there is $\alpha\in R^+$, $n\in\DZ$ with $s_{\alpha,n}x=y$, and set
$\alpha(E)=\alpha^\vee_n\otimes 1\in \hV_k$.  

Recall that we denoted by $\hw_0$ the maximal element in  $\hCW^{\text{res,$-$}}$. The following is Lemma 9.1 in \cite{fiebig-2007}.
\begin{lemma}\label{GKMprop} The moment graph $\hCG_{k,\le \hw_0}$ is a
  GKM-graph if $\ch k\ge h$.
\end{lemma}
\subsection{Multiplicities for intersection sheaves on
  $\hCG_k$}  
Let $M=\bigoplus_{n\in\DZ} M_n$ be a $\DZ$-graded $S$-module, and denote
by $M[l]$ the graded module obtained by a shift in the grading, i.e.~
such that $M[l]_n=M_{n+l}$. For a graded free $S$-module
$M\cong\bigoplus_i S[l_i]$ of finite rank set 
$$
\grk\, M:=\sum_i v^{l_i}\in\DN[v,v^{-1}].
$$

For each field $k$ and each $w\in\hCW$ we now have the corresponding
intersection sheaf $\SB_{k}(w)$ on $\hCG_k$. By construction, each of
its stalks $\SB_k(w)^x$ is a graded free $S(\hV_k)$-module of
finite rank. The following is a natural generalization of the affine
Kazhdan-Lusztig conjecture. It appears in \cite{Fie05b}.
\begin{conjecture}\label{momconj} Let $w\in\hCW$ and suppose that $k$ is such that $\hCG_{k,\le w}$ is a
  GKM-graph. Then we have
$$
\grk\, \SB_k(w)^x=v^{l(x)-l(w)}h_{x,w}
$$
for all $x\le w$.
\end{conjecture} 

If $\ch k=0$ then $\hCG_{k,\le w}$ is a GKM-graph for each $w$. 
For the applications that we have in mind, only the finitely many
objects $\SB_k(w)$ with $w\in\hCW^{\text{res,$-$}}$  play a role.  Lemma
\ref{GKMprop} ensures that for all those $w$ the graph $\hCG_{k,\le w}$
is a GKM-graph if $\ch k\ge h$.

In the following we list the proven instances of Conjecture
\ref{momconj}. Later we recall the main steps in the proof of the fact
that Conjecture \ref{momconj} implies Conjecture \ref{genconj}, and
hence Conjecture \ref{Lconj} (cf.~\cite{fiebig-2007}).

\subsection{From sheaves on the affine flag variety to intersection
  sheaves on the graph}

Suppose that $k=\DC$. In \cite{MR1871967} it is shown that the
intersection sheaves on $\hCG_\DC$ encode the torus-equivariant intersection cohomologies
of Schubert varieties in the complex affine flag variety associated
to the simply connected complex group $G^\vee$ with root system $R^\vee$. By an equivariant version of a  theorem of Kazhdan and Lusztig (cf.~\cite{MR573434}) the
graded ranks of the stalks of the intersection cohomology complexes
are given by the corresponding affine Kazhdan-Lusztig
polynomials. From these two results we can deduce Conjecture
\ref{momconj} in the case $k=\DC$. By the algebraicity and the
finiteness of the Braden--MacPherson algorithm  we can then deduce the
conjecture  for each fixed
$w$ for almost all characteristics. Hence we have the following:
\begin{theorem}\label{pbigenough} 
\begin{enumerate}
\item 
If $\ch k=0$, then Conjecture $\ref{momconj}$ holds.
\item  For each  $w\in\hCW$ the Conjecture $\ref{momconj}$
  holds if  $\ch k$ is big
  enough. (Here the notion of big enough depends on $w$.) 
\end{enumerate}
\end{theorem}
\subsection{The multiplicity one result}
One of the advantages of the moment graph picture over the geometric
picture is that one can give a characterization of its $p$-smooth locus:

\begin{theorem}[\cite{Fiebig:math0607501}]\label{multone} Fix
  $w\in\hCW$ and suppose that $k$ is such that  $\hCG_{k,\le w}$ is a GKM-graph. Then for all $x\le w$ the following are equivalent:
\begin{enumerate}
\item $\grk\, \SB_k(w)^x=1$.
\item $h_{x,w}=v^{l(w)-l(x)}$.
\end{enumerate}
\end{theorem}
The above theorem proves the multiplicity one case of Conjecture
\ref{momconj} for all relevant fields $k$.

We now apply the above results to Lusztig's conjecture.

\subsection{From intersection sheaves to $\fg$-modules}

Denote by $\CB_k^{\circ,-}\subset\Sh(\hCG_k)$ the full subcategory of
objects isomorphic to a direct sum of various intersection sheaves
$\SB_k(w)$ with $w\in\hCW^{\circ,-}$. Denote by
$\CR_k\subset\CC_{\tS}$ the full subcategory of projective objects. Here
we cite the main result of \cite{fiebig-2007}:

\begin{theorem}\label{Mth}
Suppose that $\ch
k> h$.  Then there exists an additive functor 
$$
\Phi\colon \CB_k^{\circ,-}\to
  \CR_k
$$ 
such that 
$$
\left(\Phi(\SB):Z^\prime_{\tilde S}(x\cdot 0)\right)=(\grk\, \SB^x)(1)
$$
for all $\SB\in\CB_k^{\circ,-}$ and $x\in\hCW$.
\end{theorem}
Let me point out a small, but important detail here: While
$\CB_k^{\circ,-}$ is a $\hS$-linear category, $\CR_k$ is
$\tS$-linear (recall that we identified $X^\vee\otimes_\DZ k$ with the Lie
algebra $\fh$). The
projection $X^\vee\oplus\DZ\to X^\vee$ on the first coordinate induces
a map $\hS\to S=S(\fh)\to \tS$ that is implicit in the functor $\Phi$.

We give a short sketch of the construction of $\Phi$. Suppose that $\CG$ is an
arbitrary GKM-graph with structure algebra $\CZ$. Denote by
$\CB\subset\Sh(\CG)$ the full subcategory of objects isomorphic to a
direct sum of intersection sheaves. 

In \cite{Fie05a} we
defined the  additive subcategory $\CV^{\text{ref}}\subset\CZ\catmod$ of {\em
  reflexive} modules, and we showed that the global sections of
intersection sheaves are reflexive. Hence we get a functor
$\Gamma\colon\CB\to\CV^{\text{ref}}$. Then we defined an exact structure on
$\CV^{\text{ref}}$ and showed that the image of $\Gamma$ is the full
subcategory of projective objects in $\CV^{\text{ref}}$. 

Suppose now that the moment graph is the graph $\hCG_{k,\le w}$
associated to the coroot system and some $w\in\hCW$. In \cite{Fie05b} we gave an alternative construction  of the projective
objects in $\CV^{\text{ref}}$ via translation
functors. This shows that each projective
object admits a {\em Verma flag}, i.e.~ belongs to some additive and
exact subcategory $\CV^{\text{Verma}}\subset\CV^{\text{ref}}$. Hence we can
construct our category $\Gamma(\CB_{k}^{\circ,-})$ via translation
functors. In addition, we proved that $\CB_{k}^{\circ,-}$ is equivalent to a certain subcategory of the
category of bimodules studied
by Soergel in \cite{MR2329762}. 

Set
$$
\CH_k^{\circ,-}:=\Gamma(\CB_{k}^{\circ,-})\subset\CZ_k\catmod.
$$
In
\cite{fiebig-2007} we constructed a functor $\Psi$ from 
$\CH_k^{\circ,-}$ to the combinatorial category $\CM_k$ that appears in the work of
Andersen, Jantzen and Soergel (cf.~\cite{MR1272539}). For this we use
the fact that both categories are constructed from a unit object by
repeatedly applying translation functors. One of the main results in \cite{MR1272539} is an equivalence $\DV\colon\CR_k\xrightarrow{\sim}\CM_k$ under the assumption $p>h$. 
Now our functor $\Phi$ is the composition 
$$
\Phi\colon
\CB_k^{\circ,-}\xrightarrow{\Gamma}\CH^{\circ,-}_k\xrightarrow{\Psi}\CM_k\xrightarrow{\DV^{-1}}\CR_k.
$$ 
This finishes our sketch of the construction of $\Phi$.

We are now going to apply  Theorem \ref{Mth} to Lusztig's conjecture. Fix
$w\in\hCW^{\circ,-}$. Then $\hCG_{k,\le w}$ is a GKM-graph if $\ch k\ge
h$ by Lemma \ref{GKMprop}. Suppose that Conjecture \ref{momconj} holds for
the pair $(x,w)$, i.e.~ suppose that we have $\grk\,
\SB_k(w)^x=v^{l(x)-l(w)}h_{x,w}$. Then 
consider the object $\Phi(\SB_k(w))\in\CR_k$. It surjects
onto the standard module $Z^\prime_{\tilde S}(w\cdot 0)$ since
$(\Phi(\SB_k(w))\colon Z^\prime_{\tilde S}(\mu))$ is zero for $\mu\not\ge
w\cdot 0$, and is one for $\mu=w\cdot 0$. Hence $\Phi(\SB_k(w))$
contains $Q^\prime_{\tilde S}(w\cdot 0)$ as a direct summand. 

We deduce from Theorem
\ref{Mth} that
$$
\left(Q^\prime_{\tilde S}(w\cdot 0):Z^\prime_{\tilde S}(x\cdot 0)\right) \le
\left(\Phi(\SB_k(w)):Z^\prime_{\tilde S}(x\cdot 0)\right)=h_{x,w}(1).
$$
From the results explained in  Section \ref{equivproj} and the
reciprocity principle we get
$$
[Z^\prime(x\cdot 0):L^\prime(w\cdot 0)] \le h_{x,w}(1)=p_{A_{w_0x},A_{w_0w}}(1).
$$
One can show inductively (cf.~\cite{MR1272539}, pp.238/239) that Conjecture
\ref{genconj} provides a {\em lower} bound on the respective
multiplicities, so we get:

\begin{theorem} Conjecture $\ref{momconj}$ implies Conjecture
  $\ref{genconj}$ and hence Lusztig's
  Conjecture $\ref{Lconj}$.
\end{theorem}

Now  Theorems \ref{pbigenough} and \ref{multone} yield the following
result (cf.~\cite{fiebig-2007}).
\begin{theorem}
\begin{enumerate}
\item Suppose that $\ch k$ is big enough. Then Lusztig's conjecture
  $\ref{Lconj}$ holds.
\item Suppose that $\ch k> h$. Then for all $w,x\in\hCW$ we have 
$$
[Z^\prime(x\cdot 0):L^\prime(w\cdot 0)]=1\text{ if and only if } p_{A_{w_0x},A_{w_0w}}(1)=1.
$$
\end{enumerate}
\end{theorem}
For the proof of part (2) of the above theorem
one needs, in addition to the arguments outlined above, knowledge
on whether each of the numbers on both sides is non-zero. 

\def\cprime{$'$}
\providecommand{\bysame}{\leavevmode\hbox to3em{\hrulefill}\thinspace}
\providecommand{\MR}{\relax\ifhmode\unskip\space\fi MR }
\providecommand{\MRhref}[2]{%
  \href{http://www.ams.org/mathscinet-getitem?mr=#1}{#2}
}
\providecommand{\href}[2]{#2}


\begin{thebibliography}{Hum71}

\bibitem[AJS94]{MR1272539}
H.~H.~Andersen, J.~C.~Jantzen, and W.~Soergel, \emph{Representations of quantum
  groups at a {$p$}th root of unity and of semisimple groups in characteristic
  {$p$}: independence of {$p$}}, Ast\'erisque (1994), no.~220, 321.

\bibitem[And87]{MR933346}
H.~H.~Andersen, \emph{Modular representations of algebraic groups}, The
  Arcata Conference on Representations of Finite Groups (Arcata, Calif., 1986),
  Proc. Sympos. Pure Math., vol.~47, Amer. Math. Soc., Providence, RI, 1987,
  pp.~23--36.

\bibitem[BM01]{MR1871967}
T.~Braden and R.~MacPherson, \emph{From moment graphs to intersection
  cohomology}, Math. Ann. \textbf{321} (2001), no.~3, 533--551.

\bibitem[Cli87]{MR933407}
E.~Cline, \emph{Simulating algebraic geometry with algebra. {III}. {T}he
  {L}usztig conjecture as a {$TG\sb 1$}-problem}, The Arcata Conference on
  Representations of Finite Groups (Arcata, Calif., 1986), Proc. Sympos. Pure
  Math., vol.~47, Amer. Math. Soc., Providence, RI, 1987, pp.~149--161.


\bibitem[Fie06]{Fiebig:math0607501}
P.~Fiebig, \emph{Multiplicity one results in {K}azhdan-{L}usztig theory and
  equivariant intersection cohomology}, preprint math/0607501, 2006.

\bibitem[Fie07]{fiebig-2007}
\bysame, \emph{Sheaves on affine Schubert varieties, modular representations and
  {L}usztig's conjecture}, preprint arXiv:0711.0871, 2007.

\bibitem[Fie08a]{Fie05a}
\bysame, \emph{Sheaves on moment graphs and a localization of {V}erma flags},
  Adv. Math. \textbf{217} (2008), 683--712.

\bibitem[Fie08b]{Fie05b}
\bysame, \emph{The combinatorics of {C}oxeter categories},  {T}rans. {A}mer. {M}ath. {S}oc. \textbf{360} (2008), 4211-4233.


\bibitem[Hum71]{MR0283038}
J.~E.~Humphreys, \emph{Modular representations of classical {L}ie algebras and
  semisimple groups}, J. Algebra \textbf{19} (1971), 51--79.

\bibitem[Jan03]{MR2015057}
J.~C.~Jantzen, \emph{Representations of algebraic groups}, second ed.,
  Mathematical Surveys and Monographs, vol. 107, American Mathematical Society,
  Providence, RI, 2003.

\bibitem[Kan87]{MR918844}
M.~Kaneda, \emph{On the inverse {K}azhdan-{L}usztig polynomials for
  affine {W}eyl groups}, J. Reine Angew. Math. \textbf{381} (1987), 116--135.

\bibitem[Kan88]{Kan}
\bysame, \emph{The Kazhdan-Lusztig polynomials arising in the modular
  representation theory of reductive algebraic groups}, RIMS Kokyuroku
\textbf{670} (1988), 129--162.

\bibitem[Kat85]{MR772611}
S.~Kato, \emph{On the {K}azhdan-{L}usztig polynomials for affine {W}eyl
  groups}, Adv. in Math. \textbf{55} (1985), no.~2, 103--130.

\bibitem[KL79]{MR560412}
D.~Kazhdan and G.~Lusztig, \emph{Representations of {C}oxeter groups and
  {H}ecke algebras}, Invent. Math. \textbf{53} (1979), no.~2, 165--184.

\bibitem[KL80]{MR573434}
\bysame, \emph{Schubert varieties and {P}oincar\'e duality}, Geometry of the
  Laplace operator (Proc. Sympos. Pure Math., Univ. Hawaii, Honolulu, Hawaii,
  1979), Proc. Sympos. Pure Math., XXXVI, Amer. Math. Soc., Providence, R.I.,
  1980, pp.~185--203.

\bibitem[KL93]{KLequiv}
\bysame, \emph{Tensor structures arising from affine {L}ie
  algebras. {I--IV}}, J. Amer. Math. Soc. {\bf 6} (1993) 905--947; J. Amer.
  Math. Soc. {\bf 6} (1993) 949--1011; J. Amer. Math. Soc. {\bf 7} (1994)
  335--381; J. Amer. Math. Soc. {\bf 7} (1994) 383--453 (1993).

\bibitem[KT95]{MR1317626}
M.~Kashiwara and T.~Tanisaki, \emph{Kazhdan-{L}usztig conjecture for
  affine {L}ie algebras with negative level}, Duke Math. J. \textbf{77} (1995),
  no.~1, 21--62.

\bibitem[Lus80a]{MR591724}
G.~Lusztig, \emph{Hecke algebras and {J}antzen's generic decomposition
  patterns}, Adv. in Math. \textbf{37} (1980), no.~2, 121--164.

\bibitem[Lus80b]{MR604598}
\bysame, \emph{Some problems in the representation theory of finite {C}hevalley
  groups}, The Santa Cruz Conference on Finite Groups (Univ. California, Santa
  Cruz, Calif., 1979), Proc. Sympos. Pure Math., vol.~37, Amer. Math. Soc.,
  Providence, R.I., 1980, pp.~313--317.

\bibitem[Soe95]{MR1357204}
W.~Soergel, \emph{Roots of unity and positive characteristic},
  Representations of groups (Banff, AB, 1994), CMS Conf. Proc., vol.~16, Amer.
  Math. Soc., Providence, RI, 1995, pp.~315--338.

\bibitem[Soe97]{MR1445511}
\bysame, \emph{Kazhdan-{L}usztig-{P}olynome und eine {K}ombinatorik f\"ur
  {K}ipp-{M}oduln}, Represent. Theory \textbf{1} (1997), 37--68 (electronic).

\bibitem[Soe07]{MR2329762}
\bysame, \emph{Kazhdan-{L}usztig-{P}olynome und unzerlegbare {B}imoduln \"uber
  {P}olynomringen}, J. Inst. Math. Jussieu \textbf{6} (2007), no.~3, 501--525.

\end{thebibliography}
\end{document}